\documentclass{birkmult}

 \newtheorem{thm}{Theorem}[section]
 \newtheorem{cor}[thm]{Corollary}
 
  \newtheorem{hyp}[thm]{Hypothesis}
 \newtheorem{prop}[thm]{Proposition}
 \theoremstyle{definition}
 
 \theoremstyle{remark}
 \newtheorem{rem}[thm]{Remark}
 \newtheorem*{ex}{Example}
 \numberwithin{equation}{section}

\newcommand{\R}{{\mathbb R}}

\newcommand{\N}{{\mathbb N}}

\newcommand{\A}{{\mathcal A}}

\newcommand{\tr}{\mathrm{Tr}}
\newcommand{\one}{1\!\!\!\;\mathrm{l}}
\newcommand{\eps}{\varepsilon}
\def\oo{\mathaccent23}

\begin{document}
%
%
%
%
%
%
%
%
%
\title[Compactness and asymptotic behavior]
 {Compactness and asymptotic behavior in nonautonomous linear parabolic equations with unbounded coefficients in $\R^d$}
\author[A. Lunardi]{Alessandra Lunardi}

\address{
Dipartimento di Matematica\\
Universit\`a di Parma\\
Viale G.P. Usberti, 53/A\\
43100 Parma, Italy}

\email{alessandra.lunardi@unipr.it}

\thanks{This work was partially supported by the GNAMPA reseach project 2009 "Equazioni differenziali ellittiche e paraboliche di ordine 2}



 \dedicatory{Dedicated to Herbert Amann}

\begin{abstract}
We consider a class of second order linear nonautonomous parabolic equations in $\R^d$ with time periodic unbounded coefficients. We give sufficient conditions for the evolution operator $G(t,s)$ be compact in $C_b(\R^d)$ for $t>s$, and describe the asymptotic behavior of $G(t,s)f$ as $t-s \to \infty$ in terms of a family of measures $\mu_s$, $s\in \R$, solution of the associated Fokker-Planck equation.  
\end{abstract}

\maketitle
\section{Introduction}
 
Linear nonautonomous parabolic equations in $\R^d$ are a classical subject in the mathematical literature. Most papers and books  about regular solutions are devoted to the case of bounded coefficients (e.g., \cite{LSU,friedman}, but the list is very long), and recently the interest towards unbounded coefficients grew up. The standard motivations to the study of unbounded coefficients are  on one side the well known connections with stochastic ODEs with unbounded nonlinearities, and on the other side  the changes of variables that transform bounded into unbounded coefficients, occurring in different mathematical models. 
However, only for a few equations with unbounded coefficients
 it is possible to recover the familiar results about the bounded coefficients case. Many of them exhibit very different, and at first glance surprising, aspects. Therefore, a third motivation    is the 
interest in new phenomena in PDEs. 

This paper deals with one of these new phenomena, giving sufficient conditions in order that  the  evolution operator $G(t,s)$ associated to a class of second order parabolic equations  is a  compact  contraction in $C_b(\R^d)$ for $t>s$. Precisely,   Cauchy problems   such as 
\begin{equation}
u_t(t,x) = \A(t)u(t, \cdot)(x), \quad t>s, \;x\in \R^d,
\label{equazione}
\end{equation}
\begin{equation}
u (s,x) =\varphi(x), \quad x\in \R^d
\label{Cauchy}
\end{equation}
will be considered, where the elliptic operators $\A(t)$ are defined by 
\begin{align}
\label{A(t)}
(\A(t)\varphi) (x) &: =\sum_{i,j=1}^d
q_{ij}(t,x)D_{ij}\varphi (x) + \sum_{i=1}^d b_i(t,x) D_i\varphi
(x)\notag\\
&: = \tr \left( Q(t,x) D^2 \varphi (x)\right) + \langle b(t,x),
\nabla \varphi (x) \rangle ,
\end{align}
and the (smooth enough) coefficients $q_{ij}$, $b_i$ are allowed to be unbounded. 
If $\varphi$ is smooth and it has compact support, a  classical bounded solution to  \eqref{equazione}--\eqref{Cauchy} is readily constructed, as the limit as $R\to \infty$ of the solutions $u_R$ of Cauchy-Dirichlet problems in the balls $B(0,R)$. However, classical bounded solutions  need not be unique. Under assumptions that guarantee positivity preserving in \eqref{equazione}--\eqref{Cauchy} (and hence, uniqueness of its bounded classical solution), a basic study of the evolution operator $G(t,s)$ for \eqref{equazione} in $C_b(\R^d)$ is in the paper \cite{KLL}. The evolution operator turns out to be markovian, since it has the representation
$$G(t,s)\varphi(x) = \int_{\R^d} \varphi(y) p_{t,s,x}(dy), \quad t>s, \;x\in \R^d, \;\varphi\in C_b(\R^d),$$
where the probability measures $p_{t,s,x}$ are given by $p_{t,s,x}(dy) = g(t,s,x,y)dy$ for a positive function $g$.

It is easy to see that if a markovian $G(t,s)$ is compact  in $C_b(\R^d)$, then it does not preserve $C_0(\R^d)$, the space of the continuos functions vanishing as $|x|\to \infty$, and it cannot be extended to a bounded operator in  $L^p(\R^d, dx)$ for $1\leq p<\infty$. Therefore, much of the theory developed for bounded coefficients fails. 

When a parabolic problem is not well posed in $L^p$ spaces with respect to the Lebesgue measure, it is natural to look for other measures $\mu$, and in particular to weighted Lebesgue measures, such that $G(t,s)$ acts in $L^p(\R^d, \mu)$. This is well understood in the autonomous case $\A(t)\equiv \A$, where the dynamics is held  by a semigroup $T(t)$ and $G(t,s) = T(t-s)$. Then, an important role is played by {\em invariant measures}, that are Borel probability measures $\mu$ such that 
$$\int_{\R^d} T(t)\varphi\,d\mu = \int_{\R^d}\varphi\,d\mu, \quad \varphi \in C_b(\R^d).$$
If a Markov semigroup has an invariant measure $\mu$, it can be extended in a standard way to a contraction semigroup in all the spaces $L^p(\R^d, \mu)$, $1\leq p<\infty$.  
Under broad assumptions the invariant measure is unique, and it is strongly related with the asymptotic behavior of $T(t)$, since $\lim_{t\to \infty} T(t)\varphi = \int_{\R^d}\varphi \,d\mu$, locally uniformly  if $\varphi\in C_b(\R^d)$ and in $L^p(\R^d, \mu)$ if $\varphi \in  L^p(\R^d, \mu)$, $1<p<\infty$.  
In the nonautonomous case the role of the invariant measure is played by families of measures $\{ \mu_s:\; s\in \R\}$,  called {\em evolution systems of measures}, that satisfy
\begin{equation}
\int_{\R^d} G(t,s)\varphi\,d\mu_t = \int_{\R^d}\varphi\,d\mu_s, \quad t>s, \; \varphi \in C_b(\R^d).
\label{invariance}
\end{equation}
If \eqref{invariance} is satisfied, the function $s\mapsto \mu_s$ satisfies (at least, formally) the Fokker-Planck equation
$$D_s \mu_s + \A(s)^*\mu_s = 0, \quad s\in \R, $$
which is  a parabolic equation for measures without any initial, or final, condition. Therefore it is natural to have infinitely many solutions, and to look for uniqueness of special solutions. For instance, in the autonomous case the unique stationary solution is the invariant measure, in the periodic case $\A(t) = \A(t+T)$ under reasonable assumptions there is a unique $T$-periodic solution, etc. 
Arguing as in the autonomous case, it is easy to see that if \eqref{invariance} holds then $G(t,s)$ may be extended to a contraction from $L^p(\R^d, \mu_s)$ to $L^p(\R^d, \mu_t)$ for $t>s$. Therefore, it is natural to investigate asymptotic behavior of $G(t,s)$ not only in $C_b(\R^d)$ but also in these $L^p$ spaces.

A basic study of the evolution operator for parabolic equations with (smooth enough) unbounded coefficients is in \cite{KLL}. In its sequel \cite{LLZ} we studied asymptotic behavior of $G(t,s)$ in the case of time-periodic coefficients. 

In this paper   sufficient conditions will be given for the evolution operator $G(t,s)$ be compact in $C_b(\R^d)$. Then, compactness will be used to obtain asymptotic behavior results in the case of time-periodic coefficients. Indeed, compactness implies that there exists a unique $T$-periodic evolution system of measures $\{ \mu_s:\; s\in \R\}$, and that 
denoting by $m_s\varphi $ the mean value
\begin{equation}
\label{m_s} 
m_s\varphi  := \int_{\R^d} \varphi(x) \mu_s(dx), \quad s\in \R ,
\end{equation}
 there is $\omega <0$ such that for each $\eps >0$ we have 
\begin{equation}
\label{CompAsInfty} \| G(t,s)\varphi - m_s\varphi\|_{\infty}
\leq M_{\eps} e^{(\omega +\eps) (t-s)} \|\varphi\|_{\infty} , \quad t>s,
\; \varphi \in C_b(\R^d),
\end{equation}
for some $ M_{\eps} >0$. As a consequence,  for
every $p\in (1, \infty)$ and $\eps >0$  we get
\begin{equation}
\label{CompAsExp} \| G(t,s)\varphi - m_s\varphi\|_{L^p(\R^d, \mu_t)}
\leq Me^{(\omega +\eps) (t-s)} \|\varphi\|_{L^p(\R^d, \mu_s)} , \quad t>s,
\; \varphi \in L^p(\R^d, \mu_s), 
\end{equation}
for some $M= M(p, \eps)>0$. Note that while the constant $M$ may depend on $p$, the exponential rate of decay is independent of $p$. 
These results complement the asymptotic behavior results of \cite{LLZ}, where \eqref{CompAsExp} was obtained under different assumptions.


\section{Preliminaries: the evolution operator  $G(t,s)$.}


We use standard notations. $C_b(\R^d)$ is the space of the bounded continuous functions from $\R^d$ to $\R$, endowed with the sup norm. 
$C_0(\R^d)$ is the space of the continuous functions that vanish at infinity. For $k\in \N\cup \{0\}$, $C^k_c(\R^d)$ is the space of $k$ times differentiable functions with compact support. For $0<\alpha <1$, $a<b$ and $R>0$, $C^{\alpha/2, \alpha}([a,b]\times B(0,R))$ and $C^{1+\alpha/2, 2+ \alpha}([a,b]\times B(0,R))$ are the usual parabolic H\"older spaces in the set $[a,b]\times B(0,R)$; 
$C^{\alpha/2, \alpha}_{loc}(\R^{1+d})$ and $C^{1+\alpha/2, 2+ \alpha}_{loc}(\R^{1+d})$ are the subspaces of  $C_b(\R^{1+d})$ consisting of functions whose restrictions to $[a,b]\times B(0,R)$ belong to $C^{\alpha/2, \alpha}([a,b]\times B(0,R))$ and to $C^{1+\alpha/2, 2+ \alpha}([a,b]\times B(0,R))$, respectively, for every $a<b$ and $R>0$. 

\vspace{2mm}

In this section we recall some results from \cite{KLL}  about the evolution operator for parabolic equations with unbounded coefficients. 
They were proved under standard regularity and ellipticity assumptions, and nonstandard qualitative assumptions.

\begin{hyp}
\label{hyp1} 
\begin{itemize} 
\item[(i)]
The coefficients $q_{ij}$,  $b_i$ $(i,j=1,\ldots,d)$ belong to $C^{\alpha/2, \alpha}_{loc}(\R^{1+d})$ for some $\alpha\in (0, 1)$.
\item[(ii)]
For every $(s,x)\in  \R^{1+d}$,
the matrix $Q(s,x)$ is symmetric and there exists a function
$\eta:  \R^{1+d}\to \R$ such that
$0<\eta_0:=\inf_{\R^{1+d}}\eta$ and
\begin{eqnarray*}
\langle Q(s,x)\xi , \xi \rangle \geq \eta(s,x) |\xi |^2, \qquad\;\,
\xi \in \R^d,\;\,(s,x) \in  \R^{1+d}.
\end{eqnarray*}
\item[(iii)]
There exist a positive function $W  \in C^2(\R^d)$ and a number $\lambda \in \R$ such that 
\begin{eqnarray*}
\qquad\;\, \lim_{|x|\to \infty} W (x) = \infty \quad \mbox{
and } \quad \sup_{s\in \R, \,x\in \R^d}(\A(s) W)( x) -\lambda W(x)< 0.
\end{eqnarray*}
\end{itemize}
\end{hyp}

Assumptions (i) and (ii) imply that for every $s\in \R$ and $\varphi\in C_b(\R^d)$, the Cauchy problem
\begin{equation}
\label{CP}
\left\{ \begin{array}{rcll}
\displaystyle  D_tu(t,x) &  =  & \A (t)u(t,x), & t>s,\; x \in   \R^{d},\\[2mm]
u(s,x) &  =  & \varphi(x), & x\in \R^d, 
\end{array}\right.
\end{equation}
has a bounded classical solution. 
Assumption (iii) implies that the bounded classical solution to \eqref{CP} is unique  (in fact, a maximum principle that yields uniqueness is proved in \cite{KLL} under a slightly weaker assumption). The evolution operator $G(t,s)$ is defined by 
\begin{equation}
\label{(t,s)}
G(t,s)\varphi   = u(t, \cdot), \quad t\geq s \in \R , 
\end{equation}
where $u $ is the unique bounded solution to \eqref{CP}. Some of the properties of $G(t,s)$  are in next theorem.

\begin{thm}
\label{Th:G(t,s)} 
Let Hypothesis  \ref{hyp1} hold. Define $ \Lambda
:= \{(t,s,x)\in \R^{2+d}: t > s, \;x\in \R^{d}\} $. Then:
\begin{itemize} 
\item[(i)]
for every $\varphi \in C_b(\R^d)$, the function $(t,s,x)\mapsto
G(t,s)\varphi(x)$ is continuous in $\overline{\Lambda }$. For each
$s\in \R$,   $(t,x)\mapsto G(t,s)\varphi(x)$ belongs to
$C^{1+\alpha/2, 2+\alpha}_{loc}((s, \infty) \times \R^d )$;
\item[(ii)]
for every $\varphi\in C^{2}_{c}( \R^d )$, the function
$(t,s,x)\mapsto G(t,s)\varphi(x)$ is continuously differentiable
with respect to $s$ in $\overline{\Lambda }$ and
$D_sG(t,s)\varphi(x) =   -G(t,s)\A(s)\varphi(x)$ for any
$(t,s,x)\in\overline\Lambda$;
\item[(iii)]
for each $(t,s,x)\in \Lambda$ there exists a Borel probability measure $p_{t,s,x}$ in $\R^d$ such that
\begin{equation}
\label{representation}
G(t,s)\varphi(x) = \int_{\R^d}\varphi (y)
p_{t,s,x}(dy), \quad \varphi \in C_b(\R^d).
\end{equation}
Moreover, $ p_{t,s,x}(dy) = g(t,s,x,y)dy$ for a positive function
$g$.  
\item[(iv)]
$G(t,s)$ is strong Feller;  extending it to  $L^{\infty}(\R^d, dx)$ through  formula \eqref{representation}, it maps
  $L^{\infty}(\R^d, dx)$ (and, in particular, $B_b(\R^d)$) into $C_b(\R^d)$ for $t>s$, and
\begin{eqnarray*}
\|G(t,s)\varphi \|_{\infty} \leq \|\varphi \|_{\infty}, \quad
\varphi \in L^{\infty}(\R^d, dx), \;\,t>s. 
\end{eqnarray*}
\item[(v)]  If $(\varphi_n)$ is a bounded sequence in  $C_b(\R^d)$ that converges uniformly to $\varphi $ in each compact set $K\subset\R^d$, then for each $s\in \R$ and $T>0$, $G(\cdot,s)\varphi _n$
converges to $G(\cdot , s )\varphi $ uniformly in
$[s,s+T] \times K$, for each   compact set $K \subset \R^d$.
\item[(vi)] 
For every $s\in \R$ and $R>0$, $0<\eps <T$ there is $C= C(s,\eps, T, R)>0$ such that 
$$\sup_{ s+\eps \leq t\leq s+T}\|G(t,s)\varphi \|_{C^2(B(0,R))}  \leq C\|\varphi\|_{\infty} , \quad \varphi \in C_b(\R^d).$$
\end{itemize}
\end{thm}

Statements (i) to (v) are explicitly mentioned in \cite{KLL} (Thm. 2.1, Prop. 2.4, Cor. 2.5, Prop. 3.1, Lemma 3.2). Statement (vi) is hidden in the proof of Theorem 2.1, where $G(t,s)\varphi$ is obtained by an approximation procedure, in three steps: first, for $\varphi\in C^{2+\alpha}_{c}(\R^d)$, then for 
$\varphi\in C_{0}(\R^d)$, and then for $\varphi \in C_b(\R^d)$. At each step, we have interior Schauder estimates for a sequence $u_n$ that approaches $G(t,s)\varphi$, namely for  $s\in \R$ and $R>0$, $0<\eps <T$ there is $C= C(s,\eps, T, R)>0$ such that 
$$\|u_n\|_{C^{1+\alpha/2 , 2+\alpha}([s+\eps, s+T]\times B(0,R))} \leq C\|\varphi\|_{\infty}, \quad n\in \N , $$
and $u_n$ converges to $G(t,s)\varphi$ locally uniformly. This yields (vi). 

\vspace{3mm}

 To get evolution system of measures we have to  strenghten  assumption \ref{hyp1}(iii).   The following theorem is proved in \cite{KLL}.

\begin{thm}
\label{evsyst}
Under  Hypotheses \ref{hyp1}, assume in addition that  there exist a positive function $W  \in C^2(\R^d)$ and numbers
$a$,  $c>0$    such that
\begin{equation}
\label{V}
\qquad\;\, \lim_{|x|\to \infty} W (x) = +\infty \quad \mbox{
and } \quad (\A(s) W)( x) \leq a - c W(x), \quad (s,x)\in   \R^{1+d}.
\end{equation}
Then there exists a tight$^($\footnote{i.e.,   $\forall \eps
>0$   $\exists R=R(\eps)>0$ such that
$
\mu_s(B(0,R)) \geq 1-\eps$, for all $s\in\R$.}$^)$
 evolution system of measures $\{ \mu_s: s\in
\R\}$ for $G(t,s)$.  Moreover,
\begin{equation}
\label{G(t,s)V}
G(t,s)W(x) := \int_{\R^d}W(y) p_{t,s,x}(dy) \leq
W(x) + \frac{a}{c}, \quad t>s, \;x\in \R^d,
\end{equation}
and
\begin{equation}
\label{m_tV}
  \int_{\R^d}W(y)\mu_t(dy) \leq \min W +
\frac{a}{c}, \quad t\in \R .
\end{equation}
\end{thm}

\section{Compactness in $C_b(\R^d)$}

A necessary and sufficient condition for $G(t,s)$ be compact in
$C_b(\R^d)$ for $t>s$ is very similar to the corresponding condition
in the autonomous case (\cite{MPW}).

\begin{prop}
\label{tight-compattezza} Under Hypothesis \ref{hyp1}  
the following statements are equivalent:
\begin{itemize} 
\item[(a)]
for any $t>s$, $G(t,s):  C_b(\R^d)\to C_b(\R^d)$ is compact.
\item[(b)]
for any $t>s$ the family of measures $\{p_{t,s,x}(dy):\; x\in
\R^d\}$ is tight, i.e., for every  $\eps >0$ there exists
$R=R(t,s,\eps)>0$ such that
\begin{eqnarray*}
p_{t,s,x}(B(0,R)) \geq 1-\eps, \quad x\in \R^d.
\end{eqnarray*}
\end{itemize}
\end{prop}
\begin{proof}
We follow the proof given in \cite[Prop.~ 3.6]{MPW} for the autonomous case.

Let statement (a) hold. For every $R >0$ let $\varphi_R:\R^d\to\R$
be a continuous function such that $\one _{B(0,R)} \leq
\varphi_R\leq \one _{B(0,R+1)}$. Since $\|\varphi_R\|_{\infty} \leq
1$ and $G(t,s)$ is compact, there is a  sequence $G(t,s)
\varphi_{R_n}$ that converges uniformly in the whole $\R^d$ to a
limit function $g$. Since $\varphi_R$ goes to $\one$ as $R\to
+\infty$, uniformly on each compact set and $\|\varphi_R\|_{\infty}
\leq 1$ for every $R$, by Theorem \ref{Th:G(t,s)}(v)  $\lim_{R\to \infty} G(t,s)
\varphi_R  = G(t,s)\one = \one$ uniformly on each compact set.
Then, $g\equiv 1$ and $\lim_{R\to +\infty} \|G(t,s) \varphi_R
-\one\|_{\infty} =0$. Therefore, fixed any $\eps >0$, we have
\begin{eqnarray*}
p_{t,s,x}(B(0,R) ) = (G(t,s)\one _{B(0,R)})(x)  \geq
(G(t,s)\varphi_{R})(x) \geq 1-\eps , \quad x\in\R^d,
\end{eqnarray*}
for $R$ large enough.

Let now statement (b) hold. For $t>s$ fix $r\in (s,t)$ and recall
that (see formula \eqref{representation})
\begin{eqnarray*}
(G(t,s)\varphi)(x) = (G(t,r)G(r,s)\varphi)(x)=\int_{\R^d}
(G(r,s)\varphi) (y) p_{t,r,x}(dy),
 \end{eqnarray*}
for any $\varphi\in C_b(\R^d)$ and any  $x\in \R^d$. For every $R>0$
set
\begin{eqnarray*}
(G_{R}\varphi)(x) =  \int_{B(0,R)} (G(r,s)\varphi) (y)
p_{t,r,x}(dy), \quad x\in \R^d.
\end{eqnarray*}
Each $G_R: C_b(\R^d)\to C_b(\R^d)$ is a compact operator, since it
may be written as $G_{R} = {\mathcal S}\circ {\mathcal R} \circ
G(r,s)$ where $G(r,s): C_b(\R^d)\to C_b(\R^d)$ is continuous,
${\mathcal R}: C_b(\R^d)\to C( B(0,R))$ is the
restriction operator, and $ {\mathcal S}:
C(B(0,R)) \to C_b(\R^d)$ is defined by
\begin{eqnarray*}
{\mathcal S} \psi(x) =  \int_{B(0,R)} \psi(y)  p_{t,r,x}(dy) =
(G(t,r)\widetilde{\psi})(x),\qquad\;\,x\in B(0,R),
\end{eqnarray*}
where $\widetilde{\psi}(x)$ is the null extension of $\psi$ to the
whole $\R^d$. Now, ${\mathcal R}\circ G(s,r): C_b(\R^d)\to
C(B(0,R))$ is compact by Theorem
\ref{Th:G(t,s)}(v),  
and $ {\mathcal S}$ is continuous from $C(B(0,R))$ to
$C_b(\R^d)$ because $G(t,r)$ is strong Feller by Theorem
\ref{Th:G(t,s)}(iv). 

Moreover, $G_R\to G(t,s)$ in ${\mathcal L}(C_b(\R^d))$, as
$R\to+\infty$. Indeed, for $\eps >0$ there is  $R_0>0$   such that $
p_{t,r,x}(B(0,R))\geq 1-\eps $ for each $x\in \R^d$ and $R\geq R_0$,
and consequently
\begin{eqnarray*}
|(G(t,s)\varphi)(x) - (G_R\varphi)(x) |  \leq
\|G(r,s)\varphi\|_{\infty}\int_{\R^d\setminus B(0,R)}p_{t,r,x}(dy)
\leq \eps \|\varphi\|_{\infty},
\end{eqnarray*}
for $R\geq R_0$ and for each $x\in \R^d$.

Being limit of compact operators, $G(t,s)$ is compact. \end{proof}

\begin{rem}
\label{rem:comp}
{\rm Some remarks are in order.
\begin{itemize}
\item[(i)]
An insight in the proof shows that if  $G(t,s)$ is compact for some
$t>s$, then the family   $\{p_{t,s,x}(dy):\; x\in \R^d\}$ is tight;
conversely if for some $r>s$ the family $\{p_{r,s,x}(dy):\; x\in
\R^d\}$ is tight then $G(t,s)$ is compact for each $t>r$.
\item[(ii)]
If for some $r>s$ the family $\{p_{r,s,x}(dy):\; x\in \R^d\}$ is
tight, then the family $\{p_{t,s,x}(dy):\; t\geq r, \;x\in \R^d\}$
is tight. Indeed, for every $R>0$ we have
$$\begin{array}{lll}
p_{t,s,x}(\R^d\setminus B(0,R)) &= &(G(t,s)\one _{\R^d\setminus B(0,R)})(x)
\\
\\
&=&(G(t,r)G(r,s)\one _{\R^d\setminus B(0,R)})(x)
\\
\\
&\leq &\|G(r,s)\one _{\R^d\setminus B(0,R)}\|_{\infty},
\end{array}$$
so that, if $p_{r,s,x}(\R^d\setminus B(0,R)) = G(r,s)\one
_{\R^d\setminus B(0,R)}(x)\leq \eps$ for every $x$, also
$p_{t,s,x}(\R^d\setminus B(0,R))\leq \eps$ for every $x$.
\item[(iii)]
As in the autonomous case (\cite{MPW}), if $G(t,s)$ is compact
in $C_b(\R^d)$, it does not preserve $L^p(\R^d,dx)$ for any $p\in
[1, +\infty)$ and it does not preserve $C_0(\R^d)$. Indeed, let
$R>0$ be so large that $p_{t,s,x}(B(0,R)) \geq 1/2$ for every $x\in
\R^d$, and let $\varphi \in C_c(\R^d)$ be such that $\varphi \geq
\one_{B(0,R)}$. Then,
\begin{eqnarray*}
(G(t,s)\varphi)(x) \geq (G(t,s)\one_{B(0,R)})(x) =
p_{t,s,x}(B(0,R))\geq \frac{1}{2},
\end{eqnarray*}
for every $x$, so
that $G(t,s)\varphi $ does not belong to any space $L^p(\R^d,dx)$
and to $C_0(\R^d)$.
\item[(iv)]
A similar argument shows that $\inf G(t,s)\varphi >0$ for each $t>s$ and $\varphi \in C_b(\R^d)\setminus \{0\}$, $\varphi \geq 0$. Indeed, if $\varphi(x) >0$ for each $x$, 
and $R>0$ is as before, then $(G(t,s)\varphi)(x) \geq \delta (G(t,s)\one_{B(0,R)})(x) \geq \delta/2$, with $\delta = \min_{|x|\leq R}\varphi(x) >0$. 
If $\varphi(x) \geq 0$ for each $x$, it is sufficient to recall that $G(t,s)\varphi = G(t,(s+t)/2) G((s+t)/2, s)\varphi$ and that $G((s+t)/2, s)\varphi (x)>0$ for each $x$ by Theorem \ref{Th:G(t,s)} (iii). 
\end{itemize}
}
\end{rem}

However,  to check the tightness condition of Proposition \ref{tight-compattezza} is not obvious, since the measures $p_{t,s,x}$ are not explicit, in general. In the case of time depending Ornstein-Uhlenbeck operators (e.g. \cite{DPL}), 
$$(\A(t)\varphi)(x) = \frac12 \tr  \left(Q(t)D_x^2\varphi   (x)\right)+ \langle A(t)x+f(t),\nabla \varphi(x)\rangle ,\quad x\in\R^d,$$
the measures $p_{t,s,x}$ are explicit Gaussian measures and it is possible to see that the tightness condition does not hold. Alternatively, one can check that $G(t,s)$ maps $L^p(\R^d, dx)$ into itself for every $p\in (1, \infty)$ and therefore it cannot be compact in $C_b(\R^d)$.

If the assumptions of Theorem \ref{evsyst} hold,  estimate \eqref{G(t,s)V} implies that the family $\{
p_{t,s,x}: t>s, \;x\in B(0,r)\}$ is tight for every $r>0$. However, this is not enough for compactness. 
To obtain compactness we have to strenghten condition \eqref{V} on the auxiliary function $W$.

\begin{thm}
\label{Giorgio} Let Hypotheses  \ref{hyp1}   hold.
Assume in addition that there exist a $C^2$ function $W:\R^d\mapsto \R$, such that $\lim_{|x|\to \infty} W(x) = +\infty$, 
a number 
$R>0$ and a convex increasing
function $g: [0, +\infty)\to \R$ such that   $1/g$ is in $L^1(a,
+\infty)$ for  large $a $, and
\begin{equation}
(\A (s) W)(x) \leq -g(W(x)), \qquad\;\,s\in \R ,\;\, |x| \geq R.
\label{stima-conv}
\end{equation}
Then, for every $\delta >0$ there is $C=C(\delta) >0$ such that
$(G(t,s)W)(x) \leq C$ for  every $x\in \R^d$ and  $s\leq t-\delta $.
Consequently,  the family of probabilities $\{ p_{t,s,x}(dy): s\leq
t-\delta, \,x\in \R^d\}$ is   tight, and $G(t,s)$ is compact in
$C_b(\R^d)$  for $t>s$.
\end{thm}
\begin{proof}
As a first step we show that
 \begin{equation}
 \label{maggV}
(G(t,s)W)(x)  - (G(t,r)W)(x)  \geq - \int_{r}^{s} (G(t,
\sigma)\A(\sigma) W)(x)d\sigma, \quad r<s<t, \;x\in \R^d.
\end{equation}
Let $\varphi \in C^{\infty}(\R)$ be a nonincreasing function such that $\varphi\equiv 1$ in $(-\infty , 0]$, $\varphi \equiv 0$ in $[1, + \infty)$, and define $\psi_n (t) = \int_{0}^{t} \varphi (s-n)ds$ for each $n\in \N$. The functions  $\psi_n$ are smooth and enjoy the following properties:
\begin{itemize} 
\item
$\psi_n(t) = t$ for $t \in [0, n]$,
\item
$\psi_n(t) \equiv \mathrm{const.}$ for $ t \geq n+1$,
\item
$0 \leq \psi_n' \leq 1$ and $\psi_n'' \leq 0$, 
\item
for every $t\geq 0$, the sequence $(\psi_n'(t))$ is increasing.
\end{itemize}
Then, the function $W_n := \psi_n\circ W$ belongs to
$C^2_b(\R^d)$ and it is constant outside a compact set. By Theorem 
\ref{Th:G(t,s)}(ii),  applied to $W_n - c$, we have
\begin{align*}
(G(t,s)W_n)(x) -&(G(t,r)W_n)(x) =   - \int_r^s (G(t,\sigma)\A(\sigma )W_n)(x)\, d\sigma\notag\\
&= - \int_r^s G(t,\sigma) \left \{\psi_n'(W) (\A(\sigma)
W)  + \psi_n''(W) \langle Q  \nabla W , \nabla W \rangle\right\} (x) d\sigma\notag\\
&\ge - \int_r^s G(t,\sigma)( \psi_n'(W) \A(\sigma )W)(x)d\sigma \notag\\
& = - \int_r^sd\sigma\int_{E_{\sigma}}\psi_n'(W(y))(\A(\sigma
)W)(y)p_{t,\sigma,x}(dy)
\\
&- \int_r^sd\sigma\int_{\R^d\setminus
E_{\sigma}}\psi_n'(W(y))(\A(\sigma )W)(y)p_{t,\sigma,x}(dy),
\end{align*}
where $E_{\sigma}=\{x\in\R^d: (\A(\sigma )W)(x)>0\}$. 
Letting $n\to +\infty$, the left-hand side goes to  $(G(t,s)W)(x)
-(G(t,r)W)(x)$. Concerning the right-hand side,  both integrals converge by monotone convergence. We have to prove that their limits are finite.  The first term  converges to $-\int_r^sd\sigma\int_{E_{\sigma}}(\A(\sigma
)W)(y)p_{t,\sigma,x}(dy)$, which is finite since the sets
$E_{\sigma}$ are equibounded in $\R^d$ (recall that  the function ${\mathcal A}(\sigma)W$ tends to $-\infty$
as $|x|\to +\infty$, uniformly with respect to $\sigma\in [r,s]$). The second term may be estimated 
by
\begin{align*}
&-\int_r^sd\sigma\int_{\R^d\setminus
E_{\sigma}}\psi_n'(W(y))(\A(\sigma )W)(y)p_{t,\sigma,x}(dy)\\
\le &\int_r^sd\sigma\int_{E_{\sigma}}\psi_n'(W(y))(\A(\sigma
)W)(y)p_{t,\sigma,x}(dy)+ (G(t,s)W_n)(x) -(G(t,r)W_n)(x).
\end{align*}
Letting $n\to +\infty$, we  obtain that
$\int_r^sd\sigma\int_{\R^d\setminus E_{\sigma}}(\A(\sigma
)W)(y)p_{t,\sigma,x}(dy)$ is finite. 

Summing up,  the function
$\sigma\mapsto (G(t,\sigma)(\A(\sigma)W))(x)$ is in $L^1(r,s)$ and
\eqref{maggV} follows.

Possibly replacing $g$ by $\widetilde{g} = g-C$ for a suitable
constant $C$, we may assume that  $(\A(s)W)(x) \leq -g(W(x))$ for
every $s\in \R$ and $x\in \R^d$.

Fix $x\in \R^d$, $t\in \R$, and set
\begin{eqnarray*}
\beta(s) := (G(t, t-s)W)(x), \quad s\geq 0.
\end{eqnarray*}
Then $\beta$ is measurable, since it is the  limit  of the sequence of  continuous
functions $s\mapsto G(t,t-s)W_n(x)$. Inequality \eqref{maggV} implies
$$
\beta(b)-\beta(a) \leq -\int_{t-b}^{t-a}
(G(t,\sigma)g(W))(x)d\sigma, \qquad\;\, a<b,
$$
and, since $g$ is convex,
$$\begin{array}{l}
(G(t, \sigma)g(W))(x) = \displaystyle{ \int_{\R^d} g(W(y))p_{t,\sigma, x}(dy) }
\\
\\
\geq
g\bigg(  \displaystyle{ \int_{\R^d}  W(y)p_{t,\sigma, x}(dy) } \bigg) = g((G(t,
\sigma)W)(x))
\end{array}$$
so that
\begin{equation}
\label{eq:beta}
\begin{array}{lll}
 \beta(b)-\beta(a) & \le &
 \displaystyle{-\int_{t-b}^{t-a}g((G(t,\sigma)W)(x))d\sigma }
 \\
 \\
&   = & 
 \displaystyle{-\int_{t-b}^{t-a}g(\beta(t-\sigma))d\sigma } 
 \\
 \\
 & = &  \displaystyle{ -\int_{a}^{b}
g(\beta(\sigma))d\sigma, }
\end{array}
\end{equation}
for any $a<b$. Then, for every  $s\geq 0$, $\beta(s)\leq \zeta(s)$,
where $\zeta$ is the solution of the Cauchy problem
\begin{eqnarray*}
\left\{ \begin{array}{l}
\zeta'(s) = -g(\zeta (s)), \quad s\geq 0,
\\[3mm]
\zeta(0) = W(x).
\end{array}\right.
\end{eqnarray*}
Indeed, assume by contradiction that there exists $s_0>0$ such that
$\beta(s_0)>\zeta(s_0)$, and denote by $I$ the largest interval
containing $s_0$ such that $\beta(s )>\zeta(s )$ for each $s\in
I$. 

\noindent Inequality \eqref{eq:beta} implies that $ \beta(b)-\beta(a) \leq -m (b-a)$ for $b>a$, with $m :=\min g$. In other words, the function $s\mapsto \beta(s) + ms $ is decreasing. 
This implies that  $I$ contains some left neighborhood of $s_0$. Indeed, since $s\mapsto \beta(s)+ms $ is   decreasing, then 
$$\lim_{s\to  s_0^-} \beta(s)+ms \geq \beta(s_0) + ms_0 > \zeta(s_0) + ms_0 = \lim_{s\to  s_0^-}  \zeta(s)+ms$$
so that $\lim_{s\to  s_0^-} ( \beta(s)- \zeta(s)) >0$, which yields $\beta>\zeta$ in a left neighborhood of $s_0$. 

Let    $a= \inf I $. Then $a< s_0$, and  there is a sequence $(s_n)\uparrow a$ such that $\beta(s_n) \leq \zeta(s_n) $, so that $\beta(a) + ma  \leq \lim_{n\to \infty}\beta(s_n) + ms_n  \leq \zeta(a) + ma$, that is $\beta(a) \leq \zeta(a)$. On the other hand, for each $s\in I$ we have
\begin{eqnarray*}
\beta(s) - \beta(a) \leq \int_{a}^{s} -g(\beta(\sigma))d\sigma,
\quad \zeta(s) -\zeta(a) = \int_{a}^{s} -g(\zeta(\sigma))d\sigma ,
\end{eqnarray*}
so that
\begin{eqnarray*}
\beta(s) - \zeta(s) \leq   \int_{a}^{s} [-g(\beta(\sigma))+ g(\zeta (\sigma))] \,d\sigma
, \quad s\in I.
\end{eqnarray*}
Since $\beta(\sigma) > \zeta(\sigma)$ for every $\sigma \in I$ and
$g$ is increasing, the integral in the right-hand side is
nonpositive, a contradiction. Therefore, $\beta(s)\leq \zeta(s)$ for
every $s\geq 0$. 

By standard arguments about ODE's, for every $\delta >0$ there
is 
 $C=C(\delta )$ independent on the initial datum $W(x)$ such that
$\zeta(s)\leq C$ for every $s\geq \delta$. Therefore, 
$$\beta (s)=
(G(t,t-s)W)(x)\leq C, \quad s\geq \delta ,$$
with $C$ independent of $t$. This implies that for every $\delta >0$ the family of probabilities $p_{t,s,x}(dy)$
with $s\leq t-\delta$ and $x\in \R^d$ is   tight, because for every $R>0$ we have
\begin{align*}
 p_{t,s,x}(\R^d \setminus B(0,R)) &= \int_{\R^d \setminus B(0,R)}p_{t,s,x}(dy)\\
&\le \frac{1}{\inf \{W(y):\;|y|\geq R\}}  \int_{\R^d \setminus B(0,R)}W(y)p_{t,s,x}(dy)\\
&\le \frac{1}{\inf \{W(y):\;|y|\geq R\}} (G(t,s)W)(x) \\
&\le   \frac{C}{\inf \{W(y):\;|y|\geq R\}}
\end{align*}
and $\inf \{W(y):\;|y|\geq R\} $ goes to $+\infty $ as $r\to +\infty$. So,
condition (b) of Proposition \ref{tight-compattezza} is satisfied.
\end{proof}

\begin{ex}[As in the autonomous case]
If there is $R>0$ such that 
\begin{eqnarray*}
{\rm Tr}\; Q(s,x) + \langle b(s,x),x\rangle - \frac{2}{|x|^2}
\langle Q(s,x)x, x\rangle \leq -c |x|^2(\log |x|)^{\gamma}, \quad s\in
\R ,\;\, |x|\geq R,
\end{eqnarray*}
with $c>0$ and $\gamma >1$, then the condition \eqref{stima-conv} is
satisfied by any $W$ such that $W(x) = \log |x|$ for $|x|\ge R$,
with $g(s) = cs^{\gamma}$. If the regularity and ellipticity assumptions  \ref{hyp1}(i)(ii) hold, Theorem \ref{Giorgio} implies that the evolution operator $G(t,s)$  is compact in $C_b(\R^d)$ for $t>s$. 
\end{ex}


 \section{Compactness and asymptotic behavior}


In this section we derive asymptotic behavior results from compactness of $G(t,s)$ in $C_b(\R^d)$. 

Throughout the section we assume that  Hypothesis \ref{hyp1}  holds, and that 
the coefficients $q_{ij}$ and $b_i$, $i,j=1, \ldots, d$ are periodic in time, with period $T>0$. Then the asymptotic behavior of $G(t,s)$ is driven by the spectral properties of the operators 
$$V(s) : = G(s+T, s), \quad s\in \R.$$ 
This is well known in the case of evolution operators associated to families $A(t)$ of generators of analytic semigroups, see e.g.  \cite[\S7.2]{Henry}, \cite[Ch.~6]{L}, \cite{F}. Most of the arguments are independent of analyticity and will be adapted to our situation. 

To begin with, since each $V(s)$ is a contraction in $C_b(\R^d)$, its spectrum is contained in the unit circle. Its spectral radius is $1$, since $1$ is an eigenvalue. The nonzero eigenvalues of $V(s)$ are independent of $s$, since the equality $G(t,s)V(s) = V(t)G(t,s)$ implies that for each eigenfunction $\varphi $ of $V(s)$, $G(t,s)\varphi \neq 0$ is an eigenfunction of $V(t)$ with the same eigenvalue, for $t>s$. 

 If $G(t,s)$ is compact in $C_b(\R^d)$ for $t>s$, then  $\sigma (V(s))\setminus \{0\}$ consists of isolated eigenvalues, hence it is independent of $s$. Therefore, 
\begin{equation}
\label{r}
\sup\{|\lambda| : \;\lambda\in \sigma(V(s)), \; |\lambda|<1, \;s\in \R\} := r <1.
\end{equation}
Denoting  by $Q(s)$ the spectral projection
$$Q(s) = \frac{1}{2\pi i} \int_{\partial B(0,a)} (\lambda I- V(s))^{-1}d\lambda , \quad s\in \R , $$
with any $a\in (r, 1)$,  it is not difficult to see that for every $\eps >0$ there is $M_{\eps}>0$ such that 
\begin{equation}
\label{decexp}
\|G(t,s)Q(s)\varphi  \|_{\infty} \leq M_{\eps}e^{(t-s)(\log r(s) +\eps)/T}\|\varphi\|_{\infty}, \quad t>s, \;\varphi \in C_b(\R^d). 
\end{equation}
(The proof may be obtained from the proof of  \eqref{decexp2}  in Proposition \ref{comp_as_Lp}, replacing the $L^p$ spaces considered there by $C_b(\R^d)$). 

In the proof of the next proposition we use an important corollary of the Krein-Rutman Theorem, whose proof may be found in e.g. \cite[Ch.~1]{Du}. 
 
 \begin{thm}
\label{KR}
Let $K$ be a cone with nonempty interior part $\oo{K}$ in a Banach space $X$, and let $L:X\mapsto X$ be a linear compact operator such that $L\varphi \in \oo{K}$ for each $\varphi \in K\setminus\{0\}$. Then the spectral radius $r$ of $L$ is a simple eigenvalue of $L$, and all the other eigenvalues 
have modulus $<r$. 
\end{thm}

\begin{prop}
\label{pr:compact}
If $G(t,s)$ is compact in $C_b(\R^d)$ for $t>s$, then $1$ is a simple eigenvalue of $V(s)$ for each $s$, and it is the unique  eigenvalue on the unit circle. The spectral projection $P(s) = I - Q(s)$ is given by 
$$P(s)\varphi(x) = \int_{\R^d} \varphi(y)\mu_s(dy), \quad \varphi \in C_b(\R^d), \;s\in \R , x\in \R^d, $$
where $\{ \mu_s :\;s\in \R\}$ is a $T$-periodic  evolution system of measures. 
\end{prop}
\begin{proof}
Let $K= \{\varphi \in C_b(\R^d):\; \forall x\in \R^d, \;\varphi(x) \geq 0\}$ be the cone of the nonnegative functions in $C_b(\R^d)$. By Remark \ref{rem:comp}(iv),  if $\varphi \in C_b(\R^d)\setminus \{0\}$ is such that $\varphi(x)\geq 0$ for each $x$, then 
$\inf V(s)\varphi   = \inf G(s+T,s)\varphi(x) >0$. In other words, $V(s)$ maps   $K\setminus \{0\}$ into the interior part of $K$. 
Theorem \ref{KR}  implies that the spectral radius $1$ of $V(s)$ is a simple eigenvalue, and it is  the unique eigenvalue of $V(s)$ on the unit circle. The associated spectral projection  $P(s) = I- Q(s)$, with $Q(s)$ defined above, may be expressed as
 $$P(s)\varphi = m_s\varphi \, \one , \quad \varphi\in C_b(\R^d) , $$
for some $m_s$ in the dual space of $C_b(\R^d)$. To prove that $m_s\varphi = \int_{\R^d} \varphi(y)\mu_s(dy)$ for some   measure $\mu_s$ we use the Stone--Daniell Theorem (e.g., \cite[Thm.~4.5.2]{Dudley}): it is enough to check that $m_s\varphi \geq 0$ if $\varphi \geq 0$, and that 
for each sequence  $(\varphi_n) \subset C_b(\R^d)$   such that   $(\varphi_n(x))$ is decreasing and converges to $0$ for each $x\in \R^d$, we have $\lim_{n \to \infty} m_s\varphi_n  = 0$. In this case,  $\mu_s$ is a probability measure for every $s$, because $P(s)\one = \one$. 

By the general spectral theory, $P(s) = \lim_{\lambda \to  1^-}V_{\lambda}$, where
$V_{\lambda}:=  (\lambda -1)(\lambda I - V(s))^{-1}$. In its turn,  $(\lambda I - V(s))^{-1} = \sum_{k=0}^{\infty}V(s)^k/\lambda^{k+1}$ maps nonnegative functions into nonnegative functions because  $V(s)$ does. Therefore, $m_s\varphi \, \one = P(s)\varphi \geq 0$ for each $\varphi \geq 0$. 

Let now  $\varphi_n \downarrow 0$. We claim that $V(s) \varphi_n$ converges to $0$ uniformly. Indeed, since the measures $\{ p_{s+T,s,x}:\;x\in \R^d\}$ are tight, for each $\varepsilon >0$ there is $R>0$ such that $\int_{\R^d\setminus B(0,R)}  p_{s+T,s,x}(dy) \leq \eps$, for each $x\in \R^d$. On the other hand, $\varphi_n$ converges to $0$ uniformly on $B(0,R)$ by the Dini Monotone Convergence Theorem, so that for $n$ large, say $n\geq n_0$, we have $  \varphi_n(y) \leq \eps$, for $|y|\leq R$.  Therefore, for $n\geq n_0$ we have
$$\begin{array}{lll}
0\leq   V(s)  \varphi_n(x) & = & \displaystyle{\int_{B(0,R)}  \varphi_n(y)p_{s+T,s,x}(dy) +  \int_{\R^d\setminus B(0,R)}  \varphi_n(y)p_{s+T,s,x}(dy) }
\\
\\
& \leq &\eps + \|  \varphi_n\|_{\infty} \eps 
\end{array}$$
for each $x\in \R^d$. 
Since $V(s) \varphi_n$ converges to $0$ uniformly, then $P(s)V(s)  \varphi_n = V(s) P(s) \varphi_n$ converges to $0$ uniformly. But $V(s)$ is the identity on the range of $P(s)$. Then, $P(s) \varphi_n$ converges uniformly to  $0$, which implies that  $\lim_{n \to \infty} m_s\varphi_n  = 0$. 

Let us prove that $\{ \mu_s :\;s\in \R\}$ is a $T$-periodic  evolution system of measures. Since $s\mapsto P(s)$ is $T$-periodic, then $\mu_s = \mu_{s+T}$ for each $s\in \R$. Moreover, since $V(t)G(t,s) = G(t,s)V(s)$, then $P(t)G(t,s)\varphi =G(t,s)P(s)\varphi$, for each 
$\varphi\in C_b(\R^d)$. This means   
$$\int_{\R^d} G(t,s)\varphi \,d\mu_t \one =  G(t,s) \bigg( \int_{\R^d} \varphi \, d\mu_s \,\one \bigg) , \quad \varphi\in  C_b(\R^d),$$
and since $G(t,s)\one = \one$, then 
$$\int_{\R^d} G(t,s)\varphi \,d\mu_t =   \int_{\R^d} \varphi \, d\mu_s,  \quad \varphi\in  C_b(\R^d),$$
so that $\{ \mu_s :\;s\in \R\}$ is an  evolution system of measures. \end{proof}

\begin{cor}
\label{decay_infty}
Assume that  $G(t,s)$ is compact in $C_b(\R^d)$ for $t>s$. Then:
\begin{itemize}
\item[(i)] There exists a unique $T$-periodic evolution system of measures $\{ \mu_s :\;s\in \R\}$;
\item[(ii)] Setting $\omega_0 = \log r /T $, where $r $ is defined in \eqref{r}, for each $\omega >\omega_0$ there exists $M = M(\omega)>0$ such that 
\begin{equation}
\| G(t,s)\varphi -   \int_{\R^d} \varphi \, d\mu_s\|_{\infty} \leq M  e^{\omega  (t-s)}\|\varphi\|_{\infty}, \quad t>s, \;\varphi\in C_b(\R^d),
\label{decexp1}
\end{equation}
while for $\omega<\omega_0$ there is no $M$ such that  \eqref{decexp1} holds. 
\end{itemize}
\end{cor}
\begin{proof}
Let $\{ \mu_s :\;s\in \R\}$ be the $T$-periodic evolution system of measures given by Proposition \ref{pr:compact}. Since $P(s) \varphi   =\int_{\R^d} \varphi \, d\mu_s \one$, estimate \eqref{decexp} implies \eqref{decexp1}.  Since there exist eigenvalues of $V(s) $ with modulus $r$, 
\eqref{decexp1} cannot hold for  $\omega<\omega_0$.  Indeed, if $V(s)\varphi = r\varphi $ then $P(s)\varphi  =0$ and $G(s+nT,s)\varphi = r^n\varphi = e^{\omega_0nT}\varphi$ for each $n\in \N$, so that $\| G(s+nT, s)\varphi -  m_s \varphi  \|_{\infty} = \| G(s+nT,s)\varphi    \|_{\infty} = 
e^{\omega_0nT}\|\varphi\|_{\infty}$.

If $\{ \nu_s :\;s\in \R\}$ is another $T$-periodic evolution system of measures, fix $t\in \R$ and $\varphi \in C_b(\R^d)$. Since $G(t,s)\varphi -   \int_{\R^d} \varphi  \, d\mu_s$ goes to zero uniformly as $s\to -\infty$, then 
$$ 0 = \lim_{s\to -\infty} \int_{\R^d} \bigg( G(t,s)\varphi -   \int_{\R^d} \varphi \, d\mu_s\bigg) d\nu_t =   \lim_{s\to -\infty}\bigg(  \int_{\R^d}  \varphi \, d\nu_s -   \int_{\R^d} \varphi \, d\mu_s\bigg).$$
Since $s\mapsto  \int_{\R^d}  \varphi \, d\nu_s -   \int_{\R^d} \varphi \, d\mu_s$ is $T$-periodic and goes to $0$ as $s\to -\infty$, it vanishes in $\R$. By the arbitrariness of  $\varphi$,  $\nu_s = \mu_s$ for every $s\in \R$. 
\end{proof}

Once we have an evolution system of measures $\{ \mu_s :\;s\in \R\}$, $G(t,s)$ is extendable to a contraction (still denoted by $G(t,s)$) from $L^p(\R^d, \mu_s)$ to $L^p(\R^d, \mu_t)$ for $t>s$. 
Compactness and asymptotic behavior results  in $C_b(\R^d)$ imply compactness and asymptotic behavior results in such $L^p$ spaces, as the next proposition shows.

\begin{prop}
\label{comp_as_Lp}
Let $G(t,s)$ be compact in $C_b(\R^d)$ for $t>s$. Then  for every $p\in (1, \infty)$,  $G(t,s): L^p(\R^d, \mu_s)\mapsto L^p(\R^d, \mu_t)$ is compact  for $t>s$. Moreover, for every $\omega \in (\omega_0, 0)$ and $p\in (1, \infty)$ there exist $M = M(\omega, p)>0$ such that 
\begin{equation}
\| G(t,s)\varphi -   \int_{\R^d} \varphi \, d\mu_s\|_{L^p(\R^d, \mu_t)} \leq M e^{\omega (t-s)}\|\varphi\|_{L^p(\R^d, \mu_s)}, \quad t>s, \;\varphi\in L^p(\R^d, \mu_s), 
\label{decexp2}
\end{equation}
and for every $\omega<\omega_0$ there is no $M$ such that \eqref{decexp2} holds. Here $\omega_0 $ is given by Corollary \ref{decay_infty}(ii). 
\end{prop}
\begin{proof}
Let us prove that $G(t,s): L^p(\R^d, \mu_s)\mapsto L^p(\R^d, \mu_t)$ is compact  for $t>s$. 
We have 
$G(t,s) = G(t, (t+s)/2) G((t+s)/2, s)$ where $G((t+s)/2, s)$ is bounded from $L^{\infty}(\R^d)$ to $C_b(\R^d)$ by Theorem \ref{Th:G(t,s)}(iv), and $G(t, (t+s)/2)$ is compact in $C_b(\R^d)$. Therefore, $G(t,s)$ is compact in $L^{\infty}(\R^d)$. 

Now, if $\mu_1$ and $\mu_2$ are probability measures and a linear operator is bounded from $L^1(\R^d, \mu_1)$ to $L^1(\R^d, \mu_2)$ and compact from $L^{\infty}(\R^d, \mu_1)$ to $L^{\infty}(\R^d, \mu_2)$, then it is compact from $L^p(\R^d, \mu_1)$ to $L^p(\R^d, \mu_2)$, for every $p\in (1, \infty)$ (the proof is the same as in \cite[Prop.~4.6]{MPW}, where only one probability measure was considered). 

Let us prove \eqref{decexp2}. Since   $L^{\infty}(\R^d, \mu_t)  = L^{\infty}(\R^d, \mu_s)   = L^{\infty}(\R^d, dx)$ by \cite[Prop.~5.2]{KLL}, interpolating   \eqref{decexp1}
and $\| G(t,s)\varphi -   m_s\varphi\|_{L^1(\R^d, \mu_t)} \leq 2\|\varphi\|_{L^1(\R^d , \mu_s)} $ we obtain that 
$\| G(t,s)  -  m_s  \|_{{\mathcal L}(L^p(\R^d, \mu_s), L^p(\R^d , \mu_t))} $ decays exponentially as $(t-s)\to \infty$. However, the decay rate that we obtain by interpolation depends on $p$. To prove \eqref{decexp2} it is enough to show  that the spectrum of the operators $V(s)$ in $L^p(\R^d, \mu_s)$ does not depend on $s$, and coincides with the spectrum in $C_b(\R^d)$. Since $V(s) = G(s+T,s)$, then $V(s)$ is compact  in $L^p(\R^d, \mu_s)$. Therefore, its $L^p$ spectrum (except zero) consists of eigenvalues, that are independent of $s$. 
They are independent of $p$ too, as well as the associated spectral projections, by \cite[Cor.~1.6.2]{davies}. 

The statement follows now as in the case of  evolution operators in a fixed  Banach space as in the mentioned references \cite{Henry,L,F}. Note however that our Banach spaces $L^p(\R^d, \mu_s)$ vary with $s$, so that the classical theory  cannot be used verbatim. For the reader's convenience we give the proof below.

Let  $t-s = \sigma + kT$, with $k\in \N$ and $\sigma \in [0, T)$. We have 
$$\begin{array}{l}
\|G(t,s)\varphi -   m_s \varphi \|_{L^p(\R^d, \mu_t)} = \|G(t, t-\sigma) V(s)^k (I-P(s))\varphi \|_{L^p(\R^d, \mu_t)}
\\
\\
\leq \| V(s)^k (I-P(s))\varphi \|_{L^p(\R^d, \mu_s)} = \| [V(s)(I-P(s))]^k \varphi \|_{L^p(\R^d, \mu_s)}.
\end{array}$$
For $\omega>\omega_0$ let $\eps>0$ be such that $\log (r+\eps)\leq \omega $, and let 
 $k(s)\in \N$ be such that $ \| [V(s)  (I-P(s))]^k\varphi \|_{L^p(\R^d, \mu_s)}\leq (r+\eps)^k$ for each $k > k(s)$. Therefore, if the integer part $[(t-s)/T]$ is larger than $k(s)$ we have
$$\|G(t,s)\varphi -  m_s\varphi \|_{L^p(\R^d, \mu_t)} \leq (r+\eps)^k\|\varphi \|_{L^p(\R^d, \mu_s)} 
\leq e^{(t-s) \omega} \|\varphi \|_{L^p(\R^d, \mu_s)}, $$
for each $  \varphi  \in L^p(\R^d, \mu_s)$. Using the obvious inequality  $\|G(t,s)\varphi -   m_s\varphi \|_{L^p(\R^d, \mu_t)} \leq 2 \|\varphi \|_{L^p(\R^d, \mu_s)}$ for 
$[(t-s)/T]\leq k(s)$ we arrive at 
$$\|G(t,s)\varphi -  m_s\varphi \|_{L^p(\R^d, \mu_t)} \leq M_s e^{(t-s) \omega} \|\varphi \|_{L^p(\R^d, \mu_s)}, \quad \varphi  \in L^p(\R^d, \mu_s)$$
for some $M_s>0$.  It remains to show that $M_s$ can be taken independent of $s$. Since $V$ is $T$-periodic, we may take $k(s)= k(s+T)$  and hence $M_s = M_{s+T}$ for every $s\in \R$. Therefore it is enough to show that $M_s$ can be taken independent of $s$ for $s\in [0,T)$. 
For $0\leq s<T$ and $t\geq T$ we have $m_T  G(T,s)\varphi  = \int_{\R^d} G(T,s)\varphi\,d\mu_T = m_s\varphi$,  hence
$$\begin{array}{l}
\|G(t,s)\varphi -  m_s\varphi  \|_{L^p(\R^d, \mu_t)} = \| (G(t, T)-m_T) G(T,s)\varphi  \|_{L^p(\R^d, \mu_t)}
\\
\\
\leq M_Te^{\omega(t-T)} \| G(T,s)\varphi \|_{L^p(\R^d, \mu_T)} \leq  M_Te^{|\omega| T}e^{\omega(t-s)} \|  \varphi \|_{L^p(\R^d, \mu_s)}
\end{array}$$
So, we can take $M_s = M_Te^{|\omega| T}$ for $0\leq s<T$. \eqref{decexp2} follows. 
\end{proof}


\end{document}